\documentclass[a4paper,11pt]{article}
\usepackage{amsmath, amsfonts, amscd, amssymb, amsthm, enumerate, xypic}

\def\bfB{\mathbf{B}}

\DeclareMathOperator{\Fbar}{\overline{\mathbb{F}}}

\DeclareMathOperator{\Mat}{\operatorname{M}}

\DeclareMathOperator{\Ker}{\operatorname{Ker}}
\DeclareMathOperator{\Hom}{\operatorname{Hom}}
\DeclareMathOperator{\Vect}{\operatorname{span}}
\DeclareMathOperator{\im}{\operatorname{Im}}

\DeclareMathOperator{\tr}{\operatorname{tr}}

\DeclareMathOperator{\rk}{\operatorname{rk}}

\renewcommand{\setminus}{\smallsetminus}
\renewcommand{\epsilon}{\varepsilon}

\def\F{\mathbb{F}}

\def\calM{\mathcal{M}}

\def\calS{\mathcal{S}}
\def\calT{\mathcal{T}}

\def\lcro{\mathopen{[\![}}
\def\rcro{\mathclose{]\!]}}

\theoremstyle{definition}

\theoremstyle{plain}
\newtheorem{theo}{Theorem}[section]

\newtheorem{lemma}[theo]{Lemma}

\theoremstyle{plain}

\theoremstyle{remark}

\title{A connection between Schur and Dieudonn\'e's theorems on spaces of bounded rank matrices}
\author{Cl\'ement de Seguins Pazzis\footnote{Universit\'e de Versailles Saint-Quentin-en-Yvelines, Laboratoire de Math\'ematiques
de Versailles, 45 avenue des Etats-Unis, 78035 Versailles cedex, France}
\footnote{e-mail address: dsp.prof@gmail.com}}

\begin{document}

\thispagestyle{plain}

\maketitle
\begin{abstract}
We use a double-duality argument to give a new proof of Dieudonn\'e's theorem
on spaces of singular matrices. The argument connects the situation to the structure of spaces of operators
with rank at most $1$, and works best over algebraically closed fields.
\end{abstract}




\section{Introduction}

Throughout, we consider a field $\F$ and denote by $\Mat_{n,p}(\F)$ the space of all matrices with $n$ rows, $p$ columns, and entries in $\F$,
and by $\Mat_n(\F):=\Mat_{n,n}(\F)$, whose identity matrix is denoted by $I_n$.

The study of linear subspaces of matrices with restrictive conditions on the their elements
was started in 1948 by Dieudonn\'e \cite{Dieudonne}, and latter picked up by Gerstenhaber \cite{GerstenhaberI} and Flanders \cite{Flanders}, although it
is apparent that Dieudonn\'e's article went unnoticed by these authors at the time.

Dieudonn\'e's result states the following:

\begin{theo}[Dieudonn\'e, th\'eor\`eme 1 in \cite{Dieudonne}]\label{theo:Dieudo}
Let $\calM$ be an affine subspace of $\Mat_n(\F)$ that contains no invertible matrix.
Then $\dim \calM \leq n^2-n$, and if equality holds then either
the nullspaces of the matrices of $\calM$ have a common nonzero vector, or the same is true of $\calM^T$,
or $n=|\F|=2$ and $\calM$ is equivalent to the affine space of all upper-triangular matrices with trace $1$.
\end{theo}

Considerable improvements on this result have been made in the last decades, with arbitrary upper-bounds on the rank instead of $n-1$
\cite{Flanders,Meshulam,dSPaffpres} as well as generalizations to spaces in a wide range of dimensions below the critical one
\cite{AtkLloyd,Beasley,dSPclassIsrael,dSPlargerankrevisited}. An account of the most up-to-date results on the topic can be found in the introduction of \cite{dSPlargerankrevisited}. In \cite{Dieudonne}, Theorem \ref{theo:Dieudo} was used as a lemma to determine the automorphisms of the vector space $\Mat_n(\F)$
that preserve singularity, but in our view this theorem is just as interesting as its application to one of the earliest, and most widely known, linear preserver problem.

In this note, we wish to give a new proof of Dieudonn\'e's theorem, generalized to spaces of rectangular matrices. This proof uses duality arguments to connect the problem to a folklore result, which is generally attribued to Issai Schur, and works best when the field is algebraically closed.
Let us start by restating this basic result. There and throughout $U^\star$ denotes the dual vector space of a vector space $U$, and
for vector spaces $U$ and $V$, a linear form $f \in U^\star$ and a vector $y \in V$, we write
$$f \otimes y : x \in U \mapsto f(x)\,y \in V.$$

\begin{theo}[Schur's theorem]
Let $U$ and $V$ be vector spaces, and $S$ be a nonzero linear subspace of $\Hom(U,V)$ in which every operator has rank at most $1$.
Then one of the following holds:
\begin{enumerate}[(i)]
\item There exist a non-zero linear form $f \in U^\star$ and a linear subspace $V_0$ of $V$ such that $S=f \otimes V_0$.
\item There exist a non-zero vector $y \in V$ and a linear subspace $U'_0$ of $U^\star$ such that $S=U'_0 \otimes y$.
\end{enumerate}
\end{theo}

In contrast, spaces of operators with rank at most $2$ are substantially more complicated, although their classification is regularly useful in linear algebra
and operator theory (see e.g.\ \cite{BresarSemrl}).

\section{The result}

Here is the result we wish to prove:

\begin{theo}\label{theo:main}
Let $U$ and $V$ be vector spaces with respective finite dimensions $p \geq n$.
Let $\calS$ be an affine subspace of $\Hom(U,V)$ that contains no rank $n$ operator (i.e.\ no surjective operator).
Then $\dim \calS \leq p(n-1)$. Moreover, if $\dim \calS=p(n-1)$, then:
\begin{enumerate}
\item[(i)] Either some linear hyperplane $H$ of $V$ includes the range of every operator in $\calS$;
\item[(ii)] Or $n=p$ and some non-zero vector $x \in U$ is annihilated by all the operators in $\calS$;
\item[(iii)] Or $n=p=|\F|=2$ and $\calS$ is not a linear subspace.
\end{enumerate}
\end{theo}

Note that Dieudonn\'e's theorem only allows $n=p$.

\section{The proof}

\subsection{Main tools}

Our new proof uses a traditional approach of looking at the rank $1$ operators in the translation vector space $S$ of $\calS$, and
combines it with a double-duality argument that pushes further the operator-vector duality
method that has produced many spectacular applications in the last decade \cite{dSPAtkinsontoGerstenhaber,dSPLLD1,dSPLLD2}.
The first duality is the standard trace duality: to the translation vector space $S$ we assign its trace orthogonal complement
$$S^\bot:=\{v \in \Hom(V,U) : \; \forall u \in S, \; \tr(v \circ u)=0\}.$$
To $S^\bot$, we assign the dual-operator space $\widehat{S^\bot} \subseteq \Hom(S^\bot,U)$, consisting of all the evaluation mappings
$$\widehat{y} : v \in S^\bot \mapsto v(y) \in U, \quad \text{with $y \in V$.}$$
The key to the proof will be the following: consider a non-zero vector $y \in V \setminus \{0\}$.
Then the space
$$S_{(y)}:=S \cap \Hom(U,\F y)=S \cap (U^\star \otimes y),$$
which consists of the operators in $S$ with range included in $\F y$, is precisely
$(\im \widehat{y})^\circ \otimes y$, where $(\im \widehat{y})^\circ$ stands for the dual-orthogonal of $\im \widehat{y}$ in $U^\star$.
This follows from the basic identity $\tr(a\circ (f \otimes y))=\tr(f \otimes a(y))=f(a(y))$ for all $f \in U^\star$ and $a \in \Hom(V,U)$,
and from the double-orthogonality identity $S=(S^\bot)^\bot$.
As a consequence,
\begin{equation}\label{equation:ranghaty}
\rk \widehat{y}=p-\dim S_{(y)}.
\end{equation}
This relationship will be the key to our proof of the case of equality.

Finally, to streamline the proof, we extract two basic lemmas and prove them right away:

\begin{lemma}[Substitution lemma]\label{lemma:substitution}
Let $0<n \leq p$ be positive integers.
Let $A \in \Mat_{n,p}(\F)$ have both its first row and first column nonzero.
Then one can modify all the other entries of $A$ so as to obtain a rank $n$ matrix.
Moreover, if $n<p$ then the same conclusion holds under the weaker assumption that the first row of $A$ is nonzero.
\end{lemma}

\begin{proof}
Write
$$A=\begin{bmatrix}
a & L_0 \\
C_0 & [?]_{(n-1) \times (p-1)}
\end{bmatrix} \quad \text{with $a\in \F$, $C_0 \in \F^{n-1}$ and $L_0 \in \Mat_{1,p-1}(\F)$.}$$
By a simple change of bases, we can further assume that $C_0$ and $L_0$ have all their entries zero starting from the second one (in each matrix).

If $C_0$ is zero then $a \neq 0$ and it suffices to substitute the lower-right block for $\begin{bmatrix}
I_{n-1} & [0]_{(n-1) \times (p-n)}
\end{bmatrix}$. Ditto if $L_0$ is zero.

If $C_0 \neq 0$ and $L_0 \neq 0$, then we substitute the lower-right block for
$$\begin{bmatrix}
0 & [0]_{1 \times (n-2)} & [0]_{1 \times (p-n)} \\
[0]_{(n-2) \times 1} &  I_{n-2} & [0]_{(n-2) \times (p-n)}
\end{bmatrix}.$$

Finally, if $n<p$ and the first column of $A$ is zero but not its first row, we simply substitute
the lower-right block for
$\begin{bmatrix}
[0]_{(n-1) \times 1} & I_{n-1} &  [0]_{(n-1) \times (p-n-1)}
\end{bmatrix}$.
\end{proof}

\begin{lemma}[Extraction lemma]
Let $0<n \leq p$ be positive integers.
Let
$$A=\begin{bmatrix}
K & [?]_{(n-1) \times 1} \\
[?]_{1 \times (p-1)} & ?
\end{bmatrix}\in \Mat_{n,p}(\F) \quad \text{with $K \in \Mat_{n-1,p-1}(\F)$.}$$
Denote by $E_{n,p}$ the matrix of $\Mat_{n,p}(\F)$ with exactly one nonzero entry, located at the $(n,p)$-spot, and which equals $1$.
Assume that $\rk A<n$ and $\rk (A+E_{n,p})<n$. Then $\rk K<n-1$.
\end{lemma}

\begin{proof}
This can be proved by using determinants, but a more elementary argument is easily accessible.
Assume that $\rk K=n-1$. Then the first $n-1$ rows of $A$ are linearly independent.
As $\rk A<n$ and $\rk (A+E_{n,p})<n$, we deduce that in $A$ and $A+E_{n,p}$ the last row
is a linear combination of the first $n-1$ rows of $A$. By subtracting, it follows that the last row of $E_{n,p}$ is a (nontrivial) linear combination
of the first $n-1$ rows of $A$, which would yield a vector $Y \in \F^{n-1} \setminus \{0\}$ such that $Y^TK=0$,
contradicting the assumption that $\rk K=n-1$. Hence $\rk K<n-1$.
\end{proof}

\subsection{The proof of the inequality statement and the start of the case of equality}

Now, on to the proof of the inequality statement. The proof works by induction on $p$ and $n$.
Remember that $S$ denotes the translation vector space of the affine subspace $\calS$.
Let $y \in V \setminus \{0\}$.
Noting that $\Hom(U,\F y)$ has dimension $p$, we readily find that
$$\dim S_{(y)}=0 \Rightarrow \dim \calS \leq np-p.$$
Next, assume that we have found a vector $y \in V \setminus \{0\}$ such that $0<\dim S_{(y)}<p$.
Take a linear form $f \in U^\star \setminus \{0\}$ such that $f \otimes y \in S$. Let us take a basis $(e_1,\dots,e_p)$ of $U$
in which the first vectors are annihilated by $f$ and $f(e_p)=1$, and a basis $(y_1,\dots,y_n)$ of $V$ with $y_n=y$.
Denote by $\calM$ the matrix (affine) space attached to $\calS$ in those bases, and by $\overrightarrow{\calM}$ its translation vector space. We write every matrix $M \in \Mat_{n,p}(\F)$
in block form as
$$M=\begin{bmatrix}
K(M) & C(M) \\
[?]_{1 \times (n-1)} & ?
\end{bmatrix} \quad \text{with $K(M) \in \Mat_{n-1,p-1}(\F)$ and $C(M) \in \F^{n-1}$.}$$
Note in particular that $\overrightarrow{\calM}$ contains the matrix unit $E_{n,p}$, which represents $f \otimes y$ in the said bases.
Denote by $\overrightarrow{\calM_0}$ the kernel of $M \in \overrightarrow{\calM} \mapsto (K(M),C(M))$.
Then:
\begin{itemize}
\item For all $M \in \calM$, the extraction lemma yields that $\rk K(M)<n-1$, and hence by induction
$\dim K(\calM) \leq (n-2)(p-1)$.
\item The space $\overrightarrow{\calM_0}$ is naturally isomorphic to $S_{(y)}$, so it has the same dimension.
\item Finally $\dim \calM \leq \dim \overrightarrow{\calM_0}+\dim K(\calM)+(n-1)$ by the rank theorem.
\end{itemize}
This yields
$$\dim \calS=\dim \calM \leq \dim S_{(y)}+(n-2)(p-1)+(n-1) \leq (n-1)p,$$
with equality holding only if $\dim S_{(y)}=p-1$.

The inequality statement is then obtained by noting that there is at least one vector $y \in V \setminus \{0\}$ such that $\dim S_{(y)}<p$,
otherwise $S$ would contain all the rank $1$ linear maps from $U$ to $V$, and hence all the linear maps from $U$ to $V$
because $S$ is a linear subspace of $\Hom(U,V)$, and finally $\calS=\Hom(U,V)$, contradicting the assumption that $\calS$ contains no surjective operator.

At this point, we have only used well-known arguments. The innovation comes now.
Assume that $\dim \calS=(n-1)p$, and note that $\dim S^\bot=p=\dim U$ here.
Now, from the above proof of the inequality statement, we deduce that $\dim S_{(y)}$ can only take the values $0$, $p-1$ and $p$, when $y$ ranges over $V \setminus \{0\}$.
So, we deduce from \eqref{equation:ranghaty} that in $\calT:=\widehat{S^\bot}$ the only possible ranks of the operators are $0$, $1$ and $p$.
In turn, this shows that the subset
$$\calT_1:=\{v \in \calT : \; \rk v \leq 1\}$$
is actually a linear subspace of $\calT$ unless $p=2$. And then Schur's theorem can be applied to it.

A fundamental issue is the potential existence of rank $p$ operators in $\calT$.
This difficulty can be overcome in the special case of an algebraically closed field, and it is precisely for such fields that
our proof works best. A second best situation is when $\F$ has sufficiently many elements, in which case an extension of the scalar field
to an algebraic closure allows one to reduce this situation to the previous one (Section \ref{section:redtoalgclosed}). And finally, for a field of small cardinality the idea gives a partial conclusion but it must be combined with more classical tools to conclude the proof (Section \ref{section:allfields}).

\subsection{The case of equality for an algebraically closed field}

We come back to the previous situation, and now we assume that $\F$ is algebraically closed.
As a first consequence, we observe that $\Vect(\calS)$ contains no rank $n$ operator, since the set of all operators in $\Hom(U,V)$ with rank less than $n$ is Zariski-closed
(this actually only uses the fact that $\F$ is infinite). But since $\calS \subset \Vect(\calS)$ and $\dim \Vect(\calS) \leq (n-1)p$ by the inequality statement,
we deduce that $\calS=\Vect(\calS)$, i.e.\ $\calS=S$.

Next, Schur's theorem readily yields the claimed result if $p=2$, and hence we can assume $p \geq 3$ from now on.

Now the key point, which is helped by the algebraic closeness of $\F$, is that $\calT_1$ includes a linear hyperplane of $\calT$.
Indeed, otherwise there would be two linearly independent operators $a,b$ in $\calT$ whose nontrivial linear combinations all have rank $p$.
Then $ab^{-1}$ would be an endomorphism of $U$ with no eigenvalue!
In turn, this yields a linear hyperplane $H$ of $V$ such that $\forall y \in H, \; \widehat{y} \in \calT_1$.
And now we apply Schur's theorem to $\calT_1$. There are two cases:
\begin{itemize}
\item[(A)] There is a linear subspace $W$ of $(S^\bot)^\star$ and a non-zero vector $x \in U$ such that $\calT_1=W \otimes x$.

\item[(B)] There is a linear subspace $U_0$ of $U$ and a nonzero linear form $\varphi$ on $S^\bot$ such that $\calT_1=\varphi \otimes U_0$.
\end{itemize}

Obtaining the conclusion will now be rather easy. We start with case (A).
The fact that $\widehat{y} \in \calT_1$ for all $y \in H$ yields that all the operators in $S^\bot$ map $H$ into $\F x$.
By double-orthogonality, this yields that $S$ contains all the operators that vanish at $x$ and map $U$ into $H$.
Now, we choose a basis $\bfB_1$ of $U$ in which $x$ is the first vector, and a basis $\bfB_2$ of $V$ in which the last $p-1$ vectors span $H$.
The matrix space $\calM$ that represents $S$ in those bases then contains every matrix of the form
$$\begin{bmatrix}
0 & [0]_{1 \times (p-1)} \\
[0]_{(n-1) \times 1} & [?]_{(n-1) \times (p-1)}
\end{bmatrix}.$$

Now, let $M \in \calM$. Adding an appropriate matrix of the above form to $M$, we always obtain a matrix of $\calS$
and hence a matrix with rank less than $n$. Hence the substitution lemma (Lemma \ref{lemma:substitution}) yields that
the first row of $M$ is zero, unless $n=p$ in which case the only other possibility is that the first column of $M$ is zero.

Hence, if $n<p$ we immediately conclude that all the matrices in $\calM$ have their first row zero.
Assume now that $n=p$. Then we know that every matrix in $\calM$ has its first row zero or its first column zero.
Remembering that $\calS=S$, a linearity argument yields that all the matrices in $\calM$ have their first row zero, or all have their first column zero.
To sum up:
\begin{itemize}
\item Either all the matrices in $\calM$ have their first row zero, which yields that the linear hyperplane $H$ includes the range of every operator in $\calS$;
\item Or $n=p$ and all the operators in $\calS$ vanish at $x$.
\end{itemize}
This completes the study of case (A).

\vskip 3mm
We conclude the proof by assuming that case (B) holds. This time around,
the condition that $\forall y \in H, \; \widehat{y} \in \calT_1= \varphi \otimes U_0$
has the consequence that all the operators $v \in \Ker \varphi$ vanish everywhere on the linear hyperplane $H$.
Take $\psi \in V^\star$ with kernel $H$.
As a consequence $\Ker \varphi=\psi \otimes H'$ for a linear subspace $H'$ of $U$, with $\dim H'=\dim \Ker \varphi=p-1$,
so $H'$ is a linear hyperplane of $U$.
In particular $\psi \otimes H' \subset S^\bot$, which yields that every operator in $S$ maps $H'$ into $H$
(indeed, for all $x\in H'$ and all $u \in S$ we write $0=\tr(u \circ (\psi \otimes x))=\psi(u(x))$ to obtain $u(x) \in \Ker \psi=H$).
Now, take a basis $\bfB_1$ of $U$ whose last vectors span $H'$, and a basis $\bfB_2$ of $V$ whose last vectors span $H$,
and let us represent $S$ by a space $\calM$ of matrices in those bases. Every matrix $M$ of $\calM$ then takes the simplified block form
$$M=\begin{bmatrix}
a(M) & [0]_{1 \times (p-1)} \\
C(M) & K(M)
\end{bmatrix} \quad \text{with $K(M) \in \Mat_{n-1,p-1}(\F)$,}$$
and the conclusion is near. Indeed, assume that $a$ does not vanish everywhere on $\calM$.
Then it is surjective, the set $\calS_1:=a^{-1}\{1\}$ is an affine hyperplane of $S$,
and we must have $\rk K(M)<n-1$ for all $M \in \calS_1$ (otherwise $M$ clearly has rank $n$).
Hence by the inequality statement in Theorem \ref{theo:main} we find $\dim K(\calS_1) \leq (n-2)(p-1)<(n-1)(p-1)-1$,
and then the rank theorem yields
$$\dim S \leq 1+(n-1)+\dim K(\calS_1) <(n-1)p,$$
thereby contradicting our assumptions.
Hence $a=0$ and we conclude that all the operators in $S$ have their range included in $H$.

This completes the proof for an algebraically closed field.

\subsection{Reduction to the algebraically closed case}\label{section:redtoalgclosed}

We now make the classical observation that the algebraically closed case is sufficient to handle
the case where $|\F|>n$. Let us take an algebraic closure $\Fbar$ of $\F$.

To justify the reduction to $\Fbar$, we consider the matrix version of the problem, where we take an affine subspace $\calM$ of $\Mat_{n,p}(\F)$
that contains no rank $n$ matrix and has dimension $(n-1)p$. Denote by $\overrightarrow{\calM}$ its translation vector space.
Take $A \in \calM$. The extended affine space $A+\Vect_{\Fbar}(\overrightarrow{\calM}) \subseteq \Mat_{n,p}(\Fbar)$ then has dimension $(n-1)p$ over $\Fbar$,
and now the condition $|\F|>n$ allows us to see that all its elements have rank less than $n$
(because this condition is defined by the vanishing of the $n\times n$ minors, which are polynomials of total degree $n$).
Hence, the previous case applies, and we must recover that the conclusion holds for $\calM$.
So, assume first that the first conclusion of Dieudonn\'e's theorem holds for $\calM$. Then we have a nonzero vector $Y \in \overline{\F}^n$
such that $Y^T M=0$ for all $M \in \calM$. We can take a basis $(\alpha_1,\dots,\alpha_d)$ of the $\F$-linear subspace of $\Fbar$ spanned by the entries of $Y$, and write
$Y=\sum_{k=1}^d \alpha_k Y_k$ for some $Y_1,\dots,Y_d$ in $\F^n$, all nonzero. Then $Y_k^T M=0$ for all $k \in \lcro 1,d\rcro$ and all $M \in \calM$.

If the second conclusion holds then $n=p$ and the first conclusion holds for $\overline{\calM}^T$. Then, by transposing the previous proof we obtain a vector
$X \in \F^p \setminus \{0\}$ such that $MX=0$ for all $M \in \calM$. Note that the last option in Theorem \ref{theo:main} is irrelevant here
because $|\F|>2$.

\subsection{A hybrid proof for an arbitrary field}\label{section:allfields}

In this final section, we no longer make any specific assumption on the field $\F$.
In that case we must combine our idea with more classical strategies.

In fact, we now split the discussion into two cases:
\begin{itemize}
\item[(C)] The space $\calT$ contains no rank $p$ operator.
\item[(D)] The space $\calT$ contains at least one rank $p$ operator.
\end{itemize}

So, assume that case (C) holds. Then everything is fine as we can directly apply Schur's theorem to $\calT$ itself.
And it is even easier than before because one option is ruled out in Schur's theorem for $\calT$: there cannot be a nonzero
$v \in S^\bot$ at which all the operators in $\calT$ vanish, for this would precisely mean that $v=0$!
Hence all the operators in $\calT$ map into a certain $1$-dimensional subspace $\F x$ of $U$.
This means that all the operators in $S^\bot$ map into $\F x$, and if $n<p$ this is not possible because $\dim S^\bot=p$.
Hence $n=p$. Now, by double-orthogonality we recover that $S$ contains all the operators of $\Hom(U,V)$
that vanish at $x$. Finally, assume that some $u \in \calS$ does not vanish at $x$, and extend $x$ into a basis $(x,e_2,\dots,e_p)$ of $U$
and $u(x)$ into a basis $(u(x),y_2,\dots,y_p)$ of $V$. Then $S$ contains the linear mapping $u_0$ that takes $x,e_2,\dots,e_p$
respectively to $0,y_2-u(e_2),\dots,y_p-u(e_p)$, and $u+u_0$ takes the basis $(x,e_2,\dots,e_p)$
to the basis $(u(x),y_2,\dots,y_p)$, contradicting the assumption that $\calS$ contains no rank $p$ operator.

So, all that remains is to prove the claimed conclusion when case (D) holds. However, here the duality argument seems fruitless
with no extra assumption on $\F$, and in that case we settle back to well-known techniques, which we will quickly sketch.
So, assume that case (D) holds. By the duality argument, this means that there is a non-zero vector $y \in V \setminus \{0\}$ such that $S_{(y)}=\{0\}$.
Then we take a matrix space $\calM$ that represents $\calS$, where $y$ is the last vector of the basis of $V$ under consideration.
We write the matrices of $\calM$ and of its translation vector space $\overrightarrow{\calM}$ as follows:
$$M=\begin{bmatrix}
H(M) \\
R(M)
\end{bmatrix} \quad \text{with $H(M) \in \Mat_{n-1,p}(\F)$ and $R(M) \in \Mat_{1,p}(\F)$.}$$
In $\overrightarrow{\calM}$ the vanishing of $H(M)$ implies the one of $R(M)$, as a consequence of $\calS_{(y)}=\{0\}$.
Hence $\dim H(\calM)=\dim \calM$, so $H(\calM)=\Mat_{n-1,p}(\F)$.
Hence there is an affine mapping $\varphi : \Mat_{n-1,p}(\F) \rightarrow \Mat_{1,p}(\F)$ such that
$$\forall M \in \calM, \; R(M)=\varphi(H(M)).$$
Unless $n=p=2$ and $|\F|=2$ we must prove that $\varphi$ takes the form $N \mapsto Y^T N$ for some $Y \in \F^{n-1}$.
We will entirely discard the very special case where $n=p=|\F|=2$ and $\calS$ is not a linear subspace, as it requires a
different analysis.

From there, we will only sketch the ideas as they are now classical.

We shall prove that $\varphi(N)$ always belongs to the row space of $N$.
To do this, we start by considering a special situation. Assign to every $C \in \Mat_{n-1,p-1}(\F)$
the last entry $\alpha(C)$ of the row $\varphi(\begin{bmatrix}
C & [0]_{(n-1) \times 1}
\end{bmatrix}$). This is an affine mapping that vanishes at all the rank $n-1$ matrices of $\Mat_{n-1,p-1}(\F)$.
Assume that it does not vanish everywhere. Then it is surjective onto $\F$. Take an arbitrary element $a \in \F \setminus \{0\}$
and consider $\alpha^{-1}\{a\}$, which is a linear hyperplane of $\Mat_{n-1,p-1}(\F)$ that contains no rank $n-1$ matrix.
If $p>2$, the inequality in Theorem \ref{theo:main} yields a contradiction.
Hence $p=2$, then $n=2$ and $\alpha$ is a nonconstant affine endomorphism of $\F$, and hence a bijection.
As $\alpha$ takes every nonzero element to $0$, we must have $|\F|=2$ and $\alpha : x \mapsto 1-x$, whence $\calS$ is not a linear subspace, a case we have ruled out from the start.

It follows that whenever $N \in \Mat_{n-1,p}(\F)$ has last column zero, the row $\varphi(N)$ has its last entry zero.
Performing a change of basis, this yields that whenever a linear hyperplane of $\Mat_{1,p}(\F)$
includes the row space of some matrix $N \in \Mat_{n-1,p}(\F)$, then it must contain $\varphi(N)$. Writing the row space of $N$ as an intersection of linear hyperplanes
of $\Mat_{1,p}(\F)$, we conclude that $\varphi(N)$ always belongs to the row space of $N$. In particular, $\varphi$ vanishes at the zero matrix,
and hence $\calS$ is a linear subspace of $\Mat_{n,p}(\F)$.

From there, one is reduced to the basic theorem on \emph{range-compatible} linear maps, in its transposed version
(see \cite{dSPRC1}).
Applying the transposed matrix version of proposition 2.5 of \cite{dSPRC1}, we recover a vector $Y \in \F^{n-1}$ such that
$\varphi(N)=Y^T N$ for all $N \in \Mat_{n-1,p}(\F)$. Then the extended vector
$\widetilde{Y_0}:=\begin{bmatrix}
Y \\
-1
\end{bmatrix}$ is nonzero and satisfies $\widetilde{Y_0}^T M=0$ for all $M \in \calS$. This yields the first outcome in Theorem \ref{theo:main}, thereby completing the proof.

\end{document}